\documentclass[10pt]{amsart}
\usepackage{t1enc}
\usepackage[latin2]{inputenc}
\usepackage{amsmath}
\usepackage{amsthm}
\usepackage{amssymb}
\usepackage[T1]{fontenc}
\usepackage{enumerate}
\usepackage{verbatim}
\usepackage{amscd}
\usepackage{color}

\newtheorem{thm}{Theorem}[section]
\newtheorem{defi}[thm]{Definition}
\newtheorem{lem}[thm]{Lemma}
\newtheorem{lemma}[thm]{Lemma}

\newtheorem{prop}[thm]{Proposition}
\newtheorem{conclusion}[thm]{Corollary}
\newtheorem{problem}[thm]{Problem}
\newtheorem{Question}[thm]{Question}

\newcommand{\ur}{{\mbox{ur}}}
\newcommand{\TT}{{\mathcal{T}}}
\newcommand{\rational}{\mathbb{Q}}
\newcommand{\cant}{2^{<\om}}
\newcommand{\wh}{\widehat}
\newcommand{\sub}{\subseteq}
\newcommand{\om}{\omega}
\newcommand{\eps}{\varepsilon}

\newcommand{\Borel}{\mathrm{Borel}}

\newcommand{\alg}{\mathfrak{A}}
\newcommand{\algb}{\mathfrak{B}}
\newcommand{\algc}{\mathfrak{C}}

\newcommand{\FF}{\mathcal{F}}
\newcommand{\II}{\mathcal{I}}
\newcommand{\JJ}{\mathcal{J}}
\newcommand{\PP}{\mathcal{P}}
\newcommand{\VV}{\mathcal{V}}
\newcommand{\BB}{\mathcal{B}}
\newcommand{\AAA}{\mathcal{A}}
\newcommand{\DD}{\mathcal{D}}
\newcommand{\TTT}{\mathcal{T}}

\begin{document}

\date{}
\title{On the isomorphism problem for measures on Boolean algebras}
\thanks{Piotr Borodulin--Nadzieja thanks European Science
Foundation for their support through the grant
2499 within the INFTY program.
Mirna D\v{z}amonja thanks EPSRC for their support through the grants EP/G068720 and 
EP/I00498X/1.}

\author{Piotr Borodulin--Nadzieja}
\address{Instytut Matematyczny, Uniwersytet Wroc\l awski, pl. Grunwaldzki 2/4, 50-384 Wroclaw, Poland} 
\email{pborod@math.uni.wroc.pl}
\author{Mirna D\v{z}amonja}
\address{School of Mathematics, University of East Anglia, Norwich NR4 7TJ, UK} 
\email{h020@uea.ac.uk}

\maketitle

\begin{abstract}
The paper investigates possible generalisations of Maharam's theorem to a classification of
Boolean algebras that support a finitely additive measure.
We prove that Boolean algebras that support a finitely additive non-atomic uniformly regular measure
are metrically isomorphic to a subalgebras of the Jordan algebra with the Lebesgue measure.
We give some partial analogues to be used for a classification of algebras that support 
a finitely additive non-atomic measure with a higher uniform regularity number.
We show that some naturally induced equivalence relations connected to metric isomorphism are quite 
complex even on the Cantor algebra and therefore probably we cannot hope for a nice general classification 
theorem for finitely--additive measures.

We present
an example of a Boolean algebra which supports
only separable measures but no uniformly regular
one.
\end{abstract}

\section{Introduction}\label{aims}

A celebrated result in measure theory is the theorem of Maharam \cite{Maharam} which states that if 
$\mu$ is a homogeneous $\sigma$-additive measure on a $\sigma$-complete Boolean algebra 
$\mathfrak B$, then the measure algebra of $(\mathfrak B,\mu)$ is isomorphic to the measure algebra of some $2^\kappa$ with 
the natural product measure. Moreover, every measure algebra can be decomposed into a countable sum of such algebras where the measure is homogeneous. This provides us with a very beautiful classification of $\sigma$-complete measure algebras, hence it is natural to ask if a similar characterisation can be obtained under weaker assumptions. In particular, a natural class to consider is formed by pairs 
$(\mathfrak B,\mu)$ where $\mathfrak B$ is any Boolean algebra, not necessarily $\sigma$-complete,
and $\mu$ is a {\em strictly positive finitely additive} measure on $\mathfrak B$, where $\mu$ without loss of generality assigns measure 1 to the unit element of $\mathfrak B$. In her commentaries to the 1981 edition of the ``Scottish book"
D. Maharam posed this question.

A closely connected, but not the same problem, is that of obtaining a combinatorial characterisation of
Boolean algebras $\mathfrak B$ which {\em support a measure}, that is, for which there is a 
finitely additive $\mu$ which is strictly positive on $\mathfrak B$. It is easy to see that such an algebra must satisfy the countable chain condition ccc and the question of the sufficiency of this condition was raised
by Tarski in \cite{Tarski}. Horn and Tarski in \cite{HornTarski} suggested various other chain conditions
and Gaifman in \cite{Gaifman} showed that in fact a rather strong condition of being a union 
$\mathfrak B=\bigcup_{n<\omega}\mathfrak \FF_n$ where each $\FF_{n+1}$ is $n$-linked, does not suffice 
for $\mathfrak B$ to support a measure. Kelley in \cite{Kelley} gave an exact combinatorial characterisation
of Boolean algebras that support a measure, which we therefore call Kelley algebras.
This characterisation is that $\mathfrak B\setminus \{0\}=\bigcup_{n<\omega}\mathfrak \FF_n$, where each $\FF_n$ has a positive {\em Kelley intersection number}.
The Kelley intersection number for a family $\FF$ of sets is said to be $\ge\alpha$
if for every $n<\omega$ and every sequence of $n$ elements of $\FF$ (possibly
with repetitions), there is a subsequence of length
at least $\alpha\cdot n$ which has a nonempty intersection. 
Then the Kelley intersection number is the sup of all $\alpha$
such that the intersection number of is $\ge\alpha$. This characterisation is unfortunately not very useful in practice, as it is hard to check, but nevertheless, it sheds light on our initial problem of classification. Namely, it shows that every $\sigma$-centred Boolean algebra does support a measure.\footnote{This can also be seen using the Stonre representation theorem.} The
$\sigma$-centred Boolean algebras are exactly the subalgebras of $\PP(\omega)$, which for various good reasons are considered to be unclassifiable.

This detour shows that if we hope to have a classification of Kelley algebras, we should better first restrict to some reasonable subclass. Maharam's theorem suggests that there should be some cardinal invariant
at least as a first dividing line, so the natural first reduction is to consider only those Boolean algebras that support a {\em separable} measure. This can be easily defined by noticing that a strictly positive
measure $\mu$ on a Boolean algebra $\mathfrak B$ induces a metric $d$ given by 
$d(a,b)=\mu(a\bigtriangleup b)$. This gives rise to the cardinal characteristics given by the density character of this metric space, which is exactly the Maharam type of a measure algebra if the algebra is $\sigma$-closed and the measure $\sigma$-additive and homogeneous. The measure $\mu$ is said to be {\em separable} if the density character defined above is equal to $\omega$. The question of a combinatorial characterisation of Boolean algebras that support a separable strictly positive measure has already been considered by many authors, including Talagrand in \cite{Talagrand}, who proposed a plausible candidate characterisation and showed that even so it can only be true consistently. In \cite{DzP} D\v zamonja and Plebanek have shown that there is a ZFC counterexample to this characterisation, therefore putting the characterisation programme back to zero. 

Going back to the fact that all subalgebras of $\PP(\omega)$ are Kelley algebras, we note that
it is rather easy to construct atomic separable measures: counting measures. Many subalgebras of 
$\PP(\omega)$ only support such a measure, so it is more natural to restrict our attention to the non-atomic case. For a Boolean algebra to support a non-atomic measure it is of course necessary that the algebra itself be non-atomic, so we shall mostly consider such algebras.
In \cite{DzP} it is shown that Martin's axiom (for cardinals $<\mathfrak c$) implies that every non-atomic Boolean algebra of size less than $\mathfrak c$ supports a non-atomic separable measure, which shows 
on the one hand that the characterisation of algebras supporting a non-atomic separable measure
is really about algebras of size $\mathfrak c$ (it is easy to see that such a measure cannot exists on an algebra of size $>\mathfrak c$), and on the other hand that a classification is difficult since it potentially includes all small enough non-atomic algebras.

Effros' paper \cite{Effros} enunciated and  Harrington--Kechris--Louveau paper \cite{HaKeLo} initiated a programme
of a classification of equivalence relations on Polish spaces in order to use them as a measuring device for the classification difficulty of various problems naturally arising in mathematics.  Namely, suppose that 
we wish to classify the objects in a certain class, say we are in a Polish space and we wish to classify definable subspaces of it according  to some 
equivalence relation. The equivalence relation may be understood as stating that any two equivalent objects have the same invariant. If this classification is useful, then the invariant should be definable and checking if two objects are in the same class should be doable in a definable way. If we show that such a definable classification is not possible, then we have shown that the class we started is unclassifiable in reasonable terms. Since the beginning
of this programme there has emerged a powerful machinery which has been successfully used to show the complexity of various classification problems. 
In this paper our thesis is that the class of Boolean algebras supporting a separable measure is difficult to classify.
To support it we would ideally like to use the descriptive set theoretic machinery mentioned above. However, we cannot approach this problem directly using the classification techniques of equivalence classes in Polish spaces,
since the Boolean algebras in question cannot be coded as elements of a Polish space. On the other hand, since we are after a non-classification result, we may restrict to a subclass of our initial class, which may be seen as
coming from a Polish space and show a non-classification of even that smaller class. This is exactly what we do, namely we simply consider the measures on the Cantor algebra, and we show that even in the class there the classification of measures up
to isomorphism is basically Borel complete.

The above result shows that we cannot hope for a good classification theorem for finitely--additive measures. However, in the rest
of the paper we argue that some partial classifications can be done. We cannot introduce a good invariant for classification but we can describe 
Boolean algebras with measures as subalgebras of some benchmark Boolean algebras satisfying some additional conditions. 
Restating the definition from the above, a measure $\mu$ on $\BB$ is separable iff there is a countable
$\alg\subseteq \mathfrak B$ which is {\em $\mu$-dense} in the sense that for every $b\in \BB$ and
$\varepsilon >0$ there is $a\in \alg$ such that $\mu(a\bigtriangleup b)<\varepsilon$. Let us say that 
$\alg\subseteq \mathfrak B$ is {\em $\mu$-uniformly dense} if for every $b\in \BB$ and
$\varepsilon >0$ there is $a\in \alg$ such that $a\le b$ and $\mu(b\setminus a)<\varepsilon$.
The {\em uniform density} of $(\mathfrak B, \mu)$ can be defined as the smallest cardinal 
$\kappa$ such that there is a $\mu$-uniformly dense $\alg\subseteq \mathfrak B$ of size $\kappa$.
Measures whose uniform density is $\aleph_0$ are called {\em uniformly regular} and they have
been considered in the literature, for example in \cite{Mercourakis}, where it was shown that
uniform regularity is quite different than separability. We show that in fact for the purposes of characterisation, uniformly regular measures are much superior to the separable ones, since our
Theorem \ref{jordan} gives that a Boolean algebra supports a non-atomic uniformly regular measure iff it is a subalgebra of the so called Jordan algebra, an algebra which is well known in the literature (see e.g. \cite{Jordan}, \cite{Peano} ). The Jordan algebra can be defined as a maximal 
subalgebra of the Cohen algebra, on which the Lebesgue
measure is $\sigma$--additive (on its domain). The Jordan algebra is a classical object that was introduced in the process of constructing the Jordan measure, in fact much before the Lebesgue measure. We prove that all non--atomic measures supported
by the Jordan algebra are mutually isomorphic. Thus, 
we obtain something like Maharam theorem for
uniformly regular measures: every uniformly
regular non--atomic strictly positive measure
defined on a ``maximal possible'' Boolean algebra is
isomorphic to a Lebesgue measure on the Jordan
algebra. 
[We recall in \S\ref{uniform regularity} that it is also known that algebras supporting a non-atomic separable measure are subalgebras of a fixed algebra (namely the Random one) but that there is no iff characterisation known.]
Our results suggest that the classification of non-atomic Kelley algebras should proceed with the invariant being the uniform density rather than the Maharam type. Another conclusion is that although measures on the Cantor algebra
form a complicated structure with respect to
isomorphism, they all have the same (up
isomorphism) extension to a measure 
the Jordan algebra. 

Moving on to a more general case, one should recall that in fact the separable case of Maharam's theorem was known earlier (it was used for example
in classification theorems for commutative von Neumann algebras of operators on separable Hilbert spaces)
so one can see that the countable=separable case is the first building block of the classification. We would like to use our results on the classification of uniformly regular measures as a building block for the full classification accordying to the cardinal invariant of uniform regularity.
It would be our goal to have a full analogue of Maharam's theorem. For the moment we only have partial results about the higher-dimensional case. We present them in Section \ref{higher}.

The above considerations naturally raise the question of the identity of the classes that we are considering: for example is it possible that in fact the Boolean algebras that support a uniformly regular measure are exactly those that support a separable
one? It follows from Theorem \ref{jordan}
that the Cohen algebra does not support a non-atomic uniformly non-regular measure, while it easily follows from the Kelley criterion and the fact that the Cantor algebra is dense in the Cohen algebra that the Cohen algebra does support a separable measure.
We go further and in Theorem \ref{bellike} discover a Kelley algebra which only supports separable measures
but not a uniformly regular one. 

\medskip

\textbf{Acknowledgements.} We thank Su Gao, Alain Louveau, Grzegorz Plebanek and S\l awek Solecki for useful and interesting conversations on various parts of this paper.  

\section{Basic definitions} \label{definitions}
Throughout, $\kappa$ stands for an infinite cardinal.
By a {\em measure} we understand \emph{a finitely
additive} non-negative measure. If we say that a measure $\mu$
is $\sigma$--additive on $\algb$ we do not
necessarily assume that $\algb$ is
$\sigma$--complete. We will consider only totally finite measures, so without loss of generality we only
work with probability measures. A measure on
a Boolean algebra is \emph{non--atomic} if for
every $\varepsilon>0$ there is a finite partition of
unity into elements of measures at most
$\varepsilon$.

We will denote elements of a Boolean algebra by small letters, the unity by
{\bf 1} and zero by {\bf 0}, unless we work with an algebra of sets. 

The \emph{Cantor algebra} $\alg_c$ is the only
Boolean algebra, with respect to isomorphism, which is countable and atomless.
By saying that a Boolean algebra \emph{has a Cantor subalgebra} we understand that
it has a subalgebra isomorphic to $\alg_c$.
It will be convenient to see $\alg_c$ as the Boolean algebra of
clopen subsets of $2^\om$ with the product
topology. Define \emph{the generating tree} $\TTT$
as \[ \TTT = \{[s]\colon s\in \cant\} \] and notice
that $\alg_c$ is generated by $\TTT$. We will often identify $\TTT$ with $\cant$.

The \emph{Cohen algebra} $\algc$ is the
completion of $\alg_c$. In the above setting 
$\algc$ can be seen as the algebra
${\rm Borel}(2^\om)/{\rm Meager}$. \emph{The Lebesgue measure} 
for us is the measure $\lambda$ on $\alg_c$ defined
on $\TTT$ by $\lambda([s]) =1/2^n$ if $s\in 2^n$ and
then extended to $\alg_c$ and further to $\Borel(2^\om)$. 
The \emph{Random algebra} $\mathfrak{R}$ is the Boolean algebra
${\rm Borel}(2^\om)/\{B\sub 2^\om \colon \lambda(B)=0\}$.

We will say that a Boolean algebra \emph{carries} a
measure $\mu$ if $\mu$ is defined on the whole of $\alg$. If,
additionally, $\mu$ is strictly positive on
$\alg$, we will say that $\alg$ \emph{supports} $\mu$.
By a \emph{Kelley algebra} we mean a Boolean algebra
supporting a measure.

By ${\rm Free}(\kappa)$ we denote the free algebra on $\kappa$ generators.
A Boolean algebra is \emph{big} if it has a subalgebra isomorphic to ${\rm Free}(\om_1)$. 
Boolean algebras which are not big are \emph{small}. The Cantor algebra $\alg_c$ is isomorphic to ${\rm Free}(\om)$.
\medskip

We will say that a family $\DD \sub \AAA$ is
$\mu$--\emph{dense} in $\AAA$, if 
\[ \inf\{\mu(a\bigtriangleup d)\colon d\in\DD, \
a\in \AAA\} = 0.\]

A family $\DD \sub \AAA$ is \emph{uniformly}
$\mu$--\emph{dense} in $\AAA$
if
\[ \inf\{\mu(a\setminus d)\colon d\in \DD, \ a\in
\AAA, \ d\leq a\} = 0. \]

We will say that the measure $\mu$ on $\alg$ is
\emph{separable} if there is a countable $\mu$--dense
family $\DD\sub \alg$ and the measure $\mu$ on
$\alg$ is \emph{uniformly regular} if there is a
countable uniformly $\mu$--dense family $\DD\sub
\alg$. 

Given a Boolean algebra $\alg$ with a strictly
positive measure $\mu$, we can define the {\em Fr\'echet-Nikodym
metric induced by $\mu$} by:
\[ d_\mu(a,b) = \mu(a \bigtriangleup b), \mbox{ for }a, b\in \alg.\]
Moreover, every function $f\colon \alg \to \algb$
which is a (metric) isomorphism is an isometry of 
$(\alg, d_\mu)$ and $(\algb, d_\nu)$. An
isometry $f$ between $(\alg, d_\mu)$ and $(\algb,
d_\nu)$  is not necessarily a metric isomorphism ($f({\bf 0})$ is not necessarily equal to ${\bf 0}$),
but the function $g(a) = f(a) \bigtriangleup f({\bf 0})$
is:
\begin{prop}
Suppose $\alg$ supports a measure $\mu$ and $\algb$ supports $\nu$.
If $f$ is an isometry of $(\alg, d_\mu)$ and $(\algb, d_\nu)$, then
the function $g$ defined by $g(a) = f(a) \bigtriangleup f({\bf 0})$ is a metric isomorphism of $(\alg,\mu)$ and ($\algb,\nu$).
\end{prop}

\begin{proof}
Using the fact that $\nu(g(a)) = \mu(a)$ for each $a\in \alg$ one can show that $g$ is monotonic and if $a \land b = {\bf 0}$, then $g(a) \land g(b) = {\bf 0}$. The latter implies that $g(a^c) = g(a)^c$ and $\nu(g(a\lor b)) = \nu(g(a) \lor g(b))$ for
each $a, b\in \alg$ such that $a\land b = {\bf 0}$. From the monotonicity and the latter it follows that $g(a\lor b) = g(a)\lor g(b)$ for each $a,b\in \alg$. Hence, $g$ is a metric isomorphism.
\end{proof}

Notice that a measure $\mu$ supported by $\alg$ is separable if and only if the
space $(\alg, d_\mu)$ is separable.

We will consider equivalence relations on Polish
spaces. We say that an equivalence relation ${\mathsf E}$ on
a Polish space $X$ is \emph{reducible} to an equivalence relation
${\mathsf F}$ on a Polish space $Y$ if there is a Borel
function $f\colon X\to Y$ such that
\[ x  {\mathsf E}   y \ \ \mbox{ iff } \ \ f(x)  \mathsf{F}  f(y) \]
for each $x, y\in X$.
Loosely speaking ${\mathsf E}$ is reducible to ${\mathsf F}$ if
it is not more complex than ${\mathsf F}$.
There are several benchmark equivalence relations,
which allow to place certain equivalence relation
in the complexity hierarchy (see e.g.
\cite{Gao-invariant} or \cite{Kanovei}). 
Particularly important is the notion of a 
Borel--complete (analytic--complete) equivalence
relation, i.e. such that every Borel 
(analytic) equivalence relation can be reduced to
it. An example of the former one is the
isomorphism of countable graphs. 

A Borel equivalence relation ${\mathsf E}$ is \emph{smooth} if 
it is reducible to the relation of identity on a standard Borel space, the minimal equivalence relation with respect to reducibility 
among equivalence relations with uncountably many equivalence classes.

\section{Equivalences of measures on the Cantor
algebra} \label{equivalences}

With the general aim of understanding the complexity 
of metric isomorphisms between Kelley algebras we
here study the simplest possible (non--atomic) case:  
a metric isomorphism between measures defined on the Cantor 
algebra $\alg_c$. 

We approach the question of the complexity of this relation 
by encoding measures in other mathematical objects.
As we have noticed in Section \ref{definitions}, studying isomorphisms between
measures on Boolean algebras is the same as 
investigating the isometry relation between
(Fr\' echet--Nikodym) metric spaces. 
This is promising since the isometry
of Polish spaces is a deeply explored equivalence relation.
However we cannot use here directly its theory 
since it is unclear how to recognize 
Fr\'echet--Nikodym spaces within countable 
metric spaces and how the theory of
complexity of isometries of Polish spaces 
can be used for countable metric spaces. 
We shall therefore use a different approach and study measures on the Cantor algebra.
We will encode a strictly positive measure on Cantor algebra as a
function from $2^{<\om}$ to $(0,1)$. We note that studying measures on the Cantor algebra is
not the same as studying measures on the Cantor space, although these two topics are quite related.
Any measure on the Cantor algebra induces a measure on the Cantor space. Conversely, any measure on the Cantor space can be restricted to the algebra of clopen sets and is thus related to a measure on the Cantor algebra. Measures on the Cantor space and many related concepts are studied in S. Gao's book
\cite{Gao-invariant}. \footnote{Note that some of the terminology used in \cite{} is different than the one employed here, in particular our notion of a strictly positive measure on the Cantor algebra is called `non-atomic' in
\cite{Gao-invariant}.} 

We consider three equivalence
relations  
which are strictly weaker than the
metric isomorphism. They are defined on three different subfamilies of measures on $\alg_c$ 
and they capture three different properties of measures: in Subsection \ref{trees} we deal
with strictly positive measures and the equivalence relation induced by an automorphism of $\cant$, in Subsection \ref{ideals} we 
study all measures on $\alg_c$ and the relation induced by the isomorphism of ideals on $\alg_c$ and in Subsection \ref{subsets} we
investigate strictly positive non--atomic measures and the relation induced by the equality of ranges. These
equivalence relations are in a way orthogonal to
each other, i.e. there is no inclusion between
any of their equivalence classes. None of the presented results can be used 
to show the complexity of the metric isomorphism in any of the above cases. Nevertheless, we try to 
draw some conclusions in Subsection \ref{problems}.

\subsection{Strictly positive measures on $\alg_c$ and the relation induced by an automorphism on $\cant$.}
\label{trees}

By \cite{Goncharov} (cf. \cite[Section 2]{Camerlo}),  given any isomorphic copy of the Cantor algebra,
there is a Borel procedure of finding its generating tree.  
We can identify the copy of $\alg_c$ with $\alg_c$ itself and the generating tree with $\cant$.
In this setting, every measure on $\alg_c$ is uniquely determined by the values of this measure
on $\cant$. 
Hence, every measure supported by $\alg_c$ is uniquely defined by 
an element of $(0,1)^{\cant}$.
Of course, not every function defined on
$2^{<\om}$ can be extended to a
measure on $\alg_c$ but in fact we can 
treat every element of
$(0,1)^{2^{<\om}}$ as a strictly positive measure on $\alg_c$ by using the
following Borel coding. Given $f\colon 2^{<\om} \to
(0,1)$ and $s\in 2^{<\om}$ let
\[ \mu_f(s)= \prod_{s(n)=0} f(s|n) \ \cdot
\prod_{s(n)=1} (1 - f(s|n)). \] 

So, we will treat the space of strictly
positive measures on $\alg_c$ as
$\mathsf{spM} = (0,1)^{2^{<\om}}$ with the standard topology.
It is a standard Borel space.
Every element of $\mathsf{spM}$ uniquely
defines a strictly positive measure, but a strictly positive
measure can be coded in many elements of $\mathsf{spM}$. 

We can then see the metric isomorphism 
as an equivalence relation on the standard Borel space.
Namely, we can consider the equivalence relation $\equiv$ on
$\mathsf{spM}$ induced by the metric isomorphism. 
Unfortunately, it is unclear to us how we can reduce $\equiv$ 
to an equivalence relation whose complexity is known, or which 
known equivalence relations can be reduced to $\equiv$, although it looks like $\equiv$ should be
Borel complete.
We can however say something about other equivalence
relations connected to the metric isomorphism. 

Denote by $\equiv_c$ the following equivalence
relation on $\mathsf{spM}$:
\[ f \equiv_c g \mbox{ if there is an
automorphism }\varphi \ \mbox{ of }\ \cant\mbox{ such that
}\forall s\in \cant f(s) = g(\varphi(s)). \] 
Notice that $f \equiv_c g$ implies $f \equiv
g$ (but the reverse implication does not hold). 

As $\equiv_c$ is defined by an automorphism of a compact Polish group, it follows from the known results in the theory of equivalence relations that it is a smooth Borel equivalence relation. We give a direct proof which also allows us to compare $\equiv_c$ with a well studied equivalence. Namely,  denote by $\cong$ the isomorphism of countable graphs. Note that
$\cong$ considered on the space of
finitely branching trees is a Borel
equivalence relation (cf. \cite[Theorem 1.1.3]{Harvey}).

\begin{thm}
$\equiv_c$ is Borel reducible to the isomorphism of finitely branching trees.
\end{thm}

\begin{proof}
We describe how to code an element of
$(0,1)^{\cant}$ in a finitely branching tree (in a Borel way).

Every real $r\in (0,1)$ (in fact we will
treat $r$ as an element of $2^\om$) can be coded in a
finitely branching tree, e.g. in the following
way. Start with a countable tree $\om$ with the relations
given by $n \mathsf{R} m$ if and only if $m=n+1$ and then
add a terminal node to $m$-th vertex if $r(m)=1$.

Fix a function $f\colon \cant \to (0,1)$. Start
with $T_0 = \cant$ and then stick to every
$s\in T_0$ the graph coding $f(s)$. We will denote 
the function described in this way by $\Psi$. Notice that $\Psi$ is Borel.

Denote by $\cong$ the isomorphism relation
on countable graphs. We will show that $\Psi(f)
\cong \Psi(g)$ if and only if $f \equiv_c g$. 

Assume  that there is an
isomorphism of graphs $F\colon \Psi(f) \to
\Psi(g)$. 
No vertex $s$ from $T_0$ can be mapped to a vertex
from a code of a real, since $s$ has infinitely
many pairwise \emph{disjoint} vertices below,
contrary to the vertices in the codes of reals. 
Thus, $F[T_0] = T_0$ and $F| T_0$ has to be an automorphism.
Moreover, two codes of reals are isomorphic only if they
code the same real, so finally $f \equiv_c g$.
On the other hand, an isomorphism $G\colon f \to g$ witnessing $f
\equiv_c g$ clearly induces an isomorphism between $\Psi(f)$
and $\Psi(g)$.

Thus, $\equiv_c$ is Borel reducible to $\cong$ on finitely--branching trees.
\end{proof}

\subsection{All measures on $\alg_c$ and the relation induced by the isomorphism of ideals on $\alg_c$.} \label{ideals}

We can also code measures on $\alg_c$ which
are not necessarily strictly positive. 
We can use the same formula as for strictly positive measures but 
we should consider only those functions from $\cant$ to $[0,1]$ satisfying 
the ($\Pi^0_2$) condition:
\[ \forall s\in \cant \ (f(s)=0 \implies
(f(s^{\land}0)=0
\mbox{ and } f(s^{\land}1)=0)). \]
Denote by $\mathsf{M}$ the space of elements of $[0,1]^{\cant}$ satisfying the above condition.
Notice that since $\mathsf{M}$ is a $G_\delta$ subspace of $[0,1]^{\cant}$, it
is a standard Borel space. We denote by $\equiv$ the 
equivalence relation on $\mathsf{M}$ induced by the metric isomorphism.
Denote by $=^+$ the equivalence relation on
$(0,1)^\om$ of equality of countable subsets (i.e.
$(x_n) =^+ (y_n) \mbox{ iff } \{x_n\colon
n\in\om\} = \{y_n\colon n\in \om\}$).
This is a Borel equivalence relation
strictly simpler than graph isomorphism but still quite
complex (see e.g. \cite[Section 15.3]{Gao-invariant}).

By $\equiv_m$ we mean the relation on $\mathsf{M}$ defined by 
\[ f \equiv_m g \mbox{ iff } (f(s) = 0 \iff g(s) = 0 \mbox{ for each }s\in \cant). \]
Clearly $f \equiv_m g$ if and only if the measures $\mu_f$ and $\mu_g$ are mutually absolutely continuous. This relation in the context of Cantor algebras is $\Pi^0_2$ and is actually smooth. In the
context of Cantor spaces, it
is a $\Pi^0_3$ relation and $=^+$ is reducible to $\equiv_m$ (see \cite[Lemma 8.5.5]{Gao-invariant}).

We are interested in a related equivalence relation:
\[ f \equiv_{z} g \mbox{ if there is }
g'\equiv g \mbox{ such that }
f\equiv_m g'.\]

\begin{thm} \label{non-SP}
	$\equiv_{z}$ is Borel complete.
\end{thm}

\begin{proof}
We will use the fact that the isomorphism
relation of ideals on the Cantor algebra is
Borel complete (see \cite[Theorem 4]{Camerlo}). Ideals on $\alg_c$
are isomorphic if there is an automorphism of $\alg_c$ sending one
to another. Notice that ideals on $\alg_c$ are equal if they are equal on its
generating tree.

Every proper ideal $\II$ on $\alg_c$ can be coded in a
measure. We define the measure inductively with respect to the generating tree
$\cant$. For each $s\in \cant$
\begin{itemize}
 	\item if $\mu(s) = 0$, then $\mu(s^\land 0) =
\mu(s^\land 1) = 0$;
 	\item if $\mu(s)>0$ and both
$s^\land 0$ and $s^\land 1$ are not in $\II$, then
let $\mu(s^\land 0) = \mu(s^\land 1) = 1/2 \ 
\mu(s)$;
 	\item finally, notice that if $\mu(s)>0$ and for some $l<2$, 
$s^\land l\in
\II$, then $s^\land (1-l)\notin \II$. We define
$\mu(s^\land (1-l)) = \mu(s)$ and $\mu(s\land l)=0$.
\end{itemize}

Now, encode the measure in $\mathsf{M}$ by the procedure described above.
In this way we assign 
to each ideal $\II$ on $\alg_c$ an element $f_\II$ of $\mathsf{M}$ in such a (Borel) way
that $f_\II \equiv_{z} f_\JJ$ if
and only if the ideals are isomorphic.
\end{proof}

\subsection{Strictly positive non--atomic measures on $\alg_c$ and relation induced by ranges of measures.} \label{subsets}

The basic property which allows to distinguish
measures defined on $\alg_c$ is the range. If
${\rm rng}(\mu)\ne {\rm rng}(\nu)$, then $\mu$ and $\nu$
cannot be metrically isomorphic.
So, it is tempting to code a measure in the countable 
subset of $[0,1]$ as its
range. However, there are some reasons why this
coding is not useful for our purposes.
Firstly, only particular subsets of $[0,1]$
can serve as a range of a measure. Secondly, two non--isomorphic
measures can have the same range. 

Indeed, consider $\algb\sub \alg_c$ (treated here
as ${\rm Clopen}(2^\om)$) generated by
$B_0=[(0,0)]$,
$B_1=[(0,1)]$, $B_2=[(1,0)]$ and $B_3=[(1,1)]$. Let $\mu$ and
$\nu$ be measures defined on $\algb$ in the following way. 
Let $\mu( B_0 ) 
= \mu( B_1 ) = 1/4$, $\mu( B_2 )=3/8$, $\mu(
B_3 )=1/8$. Let $\nu( B_0 ) = \nu( B_3
)= 3/8$ and $\nu( B_1 )= \nu( B_2 )=1/8$. 
Clearly, ${\rm rng}(\mu)={\rm rng}(\nu)$
and $\mu$ and $\nu$ are not isomorphic. Now,
extend $\mu$ and $\nu$ to $\mu'$ and $\nu'$
defined on $\alg_c$ in such a way that
${\rm rng}(\mu')={\rm rng}(\nu')$ and $\mu'(A)\notin \rational$ 
for every $A\in \alg_c \setminus \algb$. It is
easy to see that it can be done and that it follows
that $\mu'$ and $\nu'$ are not isomorphic.

However, there exists quite a nice coding of measures into
countable subsets of $(0,1)$, at least for non--atomic strictly positive measures.
Denote by $\mathsf{naM}$ the (Borel) subspace of $\mathsf{spM}$ consisting of
those elements of $\mathsf{spM}$ which induce non--atomic measures.
Enumerate $\cant \setminus \{\emptyset\}= \{s_n\colon n\in\om\}$ in
such a way that $|s_n| < |s_{n+1}|$ or  $|s_n| = |s_{n+1}|$
and $s_n <_{lex} s_{n+1}$ for each $n$. 
For $s\in \cant$ define
\[ \wh{s} = \bigcup\{[t]\colon
|t|=|s|\mbox{ and }t<_{lex} s\}.
\]
Notice that $\wh{\TTT} = \{\wh{s_n}\colon
n\in\om\}$ generates $\alg_c$. 
Now, define $f\colon \mathsf{naM} \to [(0,1)]^{\om}$ by 
\[ f(g)(n) = \mu_g(\wh{s_n}), \]
where $\mu_g$ is the measure induced by $g\in \mathsf{naM}$.
In other words, $r\in (0,1)$ is an element of $f(\mu)$ if
there is a level $l$ of $2^{<\om}$ and a number
$n<2^l$ such that the sum of measures of first $n$
elements of this level equals $r$ (whatever is
meant by ``first $n$ elements of level'' as long as we
have some fixed Borel procedure in mind).

Denote by $\equiv_r$ the equivalence
relation of measures on the Cantor algebra such
that $\mu \equiv_r \nu$ iff ${\rm rng}(f(\mu))={\rm rng}(f(\nu))$.

\begin{thm}
\begin{enumerate}
	\item $f$ is onto a $G_\delta$ subset of
		$[(0,1)]^\om$,
	\item $f$ is Borel,
	\item $\mu \equiv_r \nu$ implies $\mu \equiv
		\nu$ but the reverse implication does not
		hold,
	\item $\equiv_r$ is reducible to $=^+$.
\end{enumerate}
\end{thm}

\begin{proof}
To see (1) notice that a sequence is in the range 
of $f$ if it is dense in $(0,1)$ (a $G_\delta$
condition) and satisfies 
a certain more complicated, though still $G_\delta$, 
condition of ``intertwining''
($0<$ {\boldmath $x_0$} $<1$,
$0<$ {\boldmath $x_1$} $<x_0<$ {\boldmath $x_2$} $<1$,
$0<$ {\boldmath $x_3$} $<x_1<$ {\boldmath
$x_4$} $<x_0<$ {\boldmath $x_5$} $<x_2<$
{\boldmath $x_6$} $<1$ and so on).

(2) is immediate.

To prove (3) assume that ${\rm rng}(f(\mu)) =
{\rm rng}(f(\nu))$. 
We will define a metric isomorphism $\varphi$ between
$(\alg_c,\mu)$ and $(\alg_c, \nu)$.
Since $\mu$ and $\nu$ are strictly positive, they are
one--to--one on $\wh{\TTT}$ and therefore the
following definition of a measure--preserving bijection
$\wh{\varphi}\colon \wh{\TTT}\to \wh{\TTT}$ makes sense:
\[\wh{\varphi}(\wh{s}) = t \mbox{ iff }
t \in \wh{\TTT} \cap
\nu^{-1}[\mu(\wh{s})].\]
Using $\wh{\varphi}$ one can define in a
natural way a measure--preserving bijection 
$\varphi'\colon \TTT \to \TTT'$, where $\TTT'$ is
a generating tree of $\alg_c$:
\[ \varphi'(s_{n+1}) =
\wh{\varphi}(\wh{s_{n+1}}) \setminus
\wh{\varphi}(\wh{s_n}) \]
if $|s_n|=|s_{n+1}|$ and
\[ \varphi'(s_{n+1}) =
\wh{\varphi}(\wh{s_{n+1}}) \]
otherwise. Clearly $\varphi'$ has a unique
extension to an
isomorphism $\varphi\colon \alg_c \to \alg_c$
such that $\nu(\varphi(a)) = \mu(a)$ for each
$a\in \alg_c$.

To see the second part of (3) consider $\mu$ which
is one--to--one on $\alg_c$ and $\nu$ defined by
$\nu(a) = \mu(g(a))$, where $g$ is a nontrivial automorphism of
$\alg_c$. Then $\mu \equiv \nu$ but ${\rm rng}(f(\mu)) \ne
{\rm rng}(f(\nu))$. 

(1) and (2) imply that 
the equivalence relation $\equiv_r$ is reducible
to $=^+$ restricted to a $G_
\delta$ subset of $(0,1)^\om$,
which is itself reducible 
to $=^+$.
\end{proof}

\subsection{Problems} \label{problems}

We have shown that two equivalence relations similar to 
metric isomorphism of strictly positive measures on $\alg_c$ are Borel. 
However, since both of them seem to be much
simpler than the metric isomorphism, we think that 
the metric isomorphism itself is probably more 
complicated, even if considered only on strictly positive non--atomic measures 
on $\alg_c$.

\begin{problem}
What is the complexity of the equivalence relation
of the metric isomorphism of strictly positive measures on $\alg_c$?
What is the complexity of the metric isomorphism of strictly positive
non--atomic measures on $\alg_c$?
\end{problem}

Theorem \ref{non-SP} suggests that the  isomorphism of measures 
which are not necessarily strictly positive is most probably more complicated. 

\begin{problem}
What is the complexity of the equivalence relation of the metric isomorphism of all measures on $\alg_c$?
\end{problem}

Let us also indicate another possible area of research.

Superatomic Boolean algebras are known to have good
classification (they are classified by $(\alpha,
\beta)$, where $\alpha$ is the Cantor--Bendixon
height and $\beta$ is the number of atoms in the
last Cantor--Bendixon  derivative, cf.
\cite[Section 4]{Day}). Moreover, all
measures on superatomic algebras are purely
atomic, i.e. of the form
\[ \mu = \sum_{n\in\om} a_n \delta_{x_n} \]
for a sequence of real numbers $(a_n)_n$ and a
sequence of points $(x_n)_n$ in the Stone space of the
algebra. These facts motivate:
\begin{problem}
Is there a (reasonable) classification of measures
on countable superatomic Boolean algebras?
\end{problem}

\section{A characterization of uniform
regularity} \label{uniform regularity}

In this section we will prove the following theorem.\footnote{We emphasize that homomorphisms
throughout are not assumed to preserve infinitary operations.}

\begin{thm} \label{characterization}
Assume that a Boolean algebra $\alg$ supports a uniformly regular non--atomic measure $\mu$.
Then $(\alg, \mu)$ is metrically isomorphic to a subalgebra of the Jordan algebra with the
Lebesgue measure.
Consequently, a Boolean algebra supports a non--atomic
uniformly regular measure if and only if it is isomorphic to a
subalgebra of the Jordan algebra $\JJ$ containing a dense
Cantor subalgebra.
\end{thm}

At first, we show that the Boolean algebras supporting non--atomic uniformly regular measures 
can be seen as subalgebras of the Cohen algebra:

\begin{prop} \label{Cohen}
Assume that a Boolean algebra $\alg$ supports a non--atomic uniformly regular measure $\mu$.
Then $\alg$ is isomorphic to a subalgebra of the Cohen algebra.
\end{prop}

Before proving the above proposition we will show a more general fact. Recall that every finitely additive
measure on a Boolean algebra $\alg$ has a unique extension to a measure on its $\sigma$--completion $\sigma(\alg)$.

\begin{prop} \label{general-Cohen}
Assume that a Boolean algebra $\alg$ supports a measure $\mu$ and that $\algb$ is uniformly $\mu$--dense in $\alg$. 
Then $(\alg, \mu)$ is metrically isomorphic to $(\alg', \mu')$, where $\alg' \sub \sigma(\algb)$ 
and $\mu'$ is the unique extension of $\mu|_{\algb}$ to $\sigma(\algb)$.
\end{prop}
\begin{proof}
Define a function
$\varphi \colon \alg \to \algc$ in the following way:
\[ 
\varphi(a) =  \bigvee \{b \in \algb \colon b \leq a\}. 
\]

We will prove that this function is a homomorphism.
Everywhere below we assume that $b\in\algb$.
\begin{itemize}
 	\item \textbf{$\lor$}. Observe that 
 		\[ \bigvee\{b\in\algb \colon b \leq f \lor
 		g\} \geq (\bigvee\{b\in\algb\colon b\leq f\}
 		\lor \bigvee\{b\in\algb\colon b\leq
 		g\}).\] Suppose that the above inequality is
 		strict. Then, since $\algb$ is dense in
 		$\sigma(\algb)$, we
 		would find nonzero $x\in \algb$ such that
 		\[ x \leq \varphi(f\lor g) \setminus (\varphi(f) \lor
 		\varphi(g)). \]
		We have $x\land f\ne {\bf 0}$ or $x\land
		g\ne {\bf 0}$ (otherwise, we would have
		$x \land a = {\bf 0}$ for every $a\leq f\lor g$, 
		but this is impossible since $x \leq \varphi(f\lor g)$). 
		Say $x\land f\ne {\bf 0}$.
 		As $\algb$ is dense in $\alg$, there is
 		$b\in\algb$ such that
 		$b\leq x\land f$. Then
 		$b \leq f$ and $b \leq x$, a contradiction.
 	\item \textbf{$^c$}. Similarly, assume that
 		\[ \bigvee\{b'\in \algb\colon b'\leq f^c\}
 		\ne \bigwedge\{b^c\colon b\leq f\}\]
		for some $f\in \alg$.
 		That would mean that there are $b, b'\in \algb$ with
 		$b\leq f$ and $b'\leq f^c$
 		such that either
 		\[ b'\nleq b^c,\]
		or
		\[ b^c \nleq b' \]
 		which is impossible.
\end{itemize}
Since $\algb$ is uniformly $\mu$--dense in $\alg$, we have
\[ \mu(a) = \mu'(\varphi(a)) \]
for each $a\in \alg$. Clearly, $\varphi$ is one--to--one, so $\varphi$ is a metric monomorphism.
\end{proof}

From the above proposition we can deduce Proposition \ref{Cohen}. 
Indeed, if $\alg$ supports a uniformly regular non--atomic measure,
then it has a dense countable family. We can
assume that this family is a subalgebra $\algb$ of $\alg$.
It is necessarily non--atomic and, therefore,
isomorphic to the Cantor algebra $\alg_c$. 
We can thus assume that $\alg_c$ is a dense subalgebra of $\alg$.
But the $\sigma$--completion of $\alg_c$ is the Cohen algebra. 

We will strengthen Proposition \ref{Cohen}.  
First, we introduce a general definition. Let $\mu$ be a measure defined on $\algb$. We
say that the algebra
\[ \JJ_\mu(\algb) = \{a\in \sigma(\algb)\colon \mu_*(a) = \mu^*(a)\} \]
is \emph{a $\mu$--Jordan extension of $\algb$}. Equivalently we can say that $\JJ_\mu(\algb)$ consists of those
Baire subsets of $\mathrm{Stone}(\algb)$ whose boundaries are $\wh{\mu}$--null (where $\wh{\mu}$ is the unique
extension of $\mu$). The measure $\wh{\mu}$ is strictly positive on $\JJ_\mu(\algb)$ and $\algb$ is uniformly $\wh{\mu}$--dense in $\JJ_\mu(\algb)$.
Notice also that $\algb$ is not uniformly $\wh{\mu}$--dense in any subalgebra of $\sigma(\algb)$ which is not 
included in $\JJ_\mu(\algb)$. 

Coming back to uniformly regular measures, the above remarks imply the following generalization of Proposition \ref{Cohen}:

\begin{prop}\label{somemu} Every Boolean algebra supporting a
uniformly regular measure is metrically isomorphic to
a subalgebra of $\JJ_\mu(\alg_c)$ for some $\mu$.
\end{prop}

With $\lambda$ being the Lebesgue measure, the $\lambda$-Jordan extension of $\alg_c$ is in fact the well known Jordan algebra $\JJ$ coming from the
Jordan measure on [0,1], see \cite{Jordan} and \cite{Peano}.

In Theorem \ref{jordan} we will generalize Proposition \ref{somemu}
further by showing that $\JJ_\mu(\alg_c)$, which we will denote from now on by $\JJ_\mu$, is
the same algebra for every non--atomic measure $\mu$ on $\alg_c$ with respect to isomorphism. 
So, in fact every $\JJ_\mu$ is simply the Jordan algebra
$\JJ$. Let us proceed towards the proof.

In our setting all elements of the Cohen algebra are countable
unions of elements of the generating tree $\TTT$. The key tool in
verifying if a given element of $\algc$ is in
$\JJ_\mu$ are the partitions of $\textbf{1}$ into
elements of $\TTT$.

\begin{defi}
We will say that a partition $\{a_n\colon
n\in\om\}\sub \TTT$ of $\textbf{1}$ is \emph{good for $\mu$} (or,
shortly, $\mu$--good) if
\[ \sum_n \mu(a_n) = 1. \]
\end{defi}

\begin{prop} \label{good partitions}
Consider a measure $\mu$ on the Cantor algebra
$\alg_c$. If a partition $(a_n)_n$ is good for $\mu$, then for every
$M\sub \om$
\[\bigvee_{n\in M} a_n \in \JJ_\mu.\]
If $(a_n)_n$ is not good for $\mu$, then for every
infinite co--infinite $M\sub \om$ 
\[\bigvee_{n\in M} a_n \notin \JJ_\mu. \]
\end{prop}

\begin{proof} Suppose first that $(a_n)_n$ is good for $\mu$.
Notice then that for any element $c\in \algc$ we have
\[ \mu_*(c) = \sup\{\sum_{n\in \om} \mu(b_n)\colon b_n \in
\alg_c, \ (b_n)_n \mbox{ is pairwise disjoint}, \
\bigvee b_n = c \}. \]
To show that $\mu_*(c) = \mu^*(c)$ we have to
prove that $\mu_*(c)+\mu_*(c^c)=1$, but if $c=\bigvee_{n\in M} a_n$,
then 
\[ 1\ge \mu_*(c) + \mu_*(c^c) \ge \sum_{n\in\om}
\mu(a_n) = 1, \]
since $(a_n)_n$ is good.

The second part of the proposition can be proved
in the same way.
\end{proof}

Notice that for every partition $\{a_n\colon n\in\om\}
\sub \TTT$ of $\textbf{1}$ there are measures $\mu$,
$\nu$, for which $(a_n)_n$ is $\mu$--good and not
$\nu$--good. Indeed, 
define the measures on $(a_n)$ by $\mu(a_n)=1/2^{n+1}$ and
$\nu(a_n)=1/2^{n+2}$ for every $n$ and then extend them to strictly positive and
non--atomic measures on $\alg_c$. 
It is also easy to see that for every measure $\mu$ on
$\alg_c$ and $\eps>0$ there is a partition $(a_n)_n$
of $\textbf{1}$ such that $\Sigma_n \mu(a_n) < \eps$. In particular,
there is no measure such that $\JJ_\mu = \algc$
and therefore there is no strictly positive uniformly
regular measure on the Cohen algebra. Of course, the Cohen algebra is Kelley, 
so we have the following

\begin{conclusion}
There are Boolean algebras with a countable dense
set (and, thus, Kelley) which do not support a
uniformly regular measure.
\end{conclusion}

Now we will show that a Boolean algebra $\JJ_\mu$ is
isomorphic to $\JJ_\lambda$ for every non--atomic
measure $\mu$. In analogy with what we have done with partitions of unity,
for a set $x\in \JJ_\mu$ we say that $\{a_n\colon n\in\om\} \sub
\TTT$ is a {\em $\mu$--good
partition of $x$} if $\sum_n \mu(a_n) = \wh{\mu}(x)$,
where $\wh{\mu}$ is the unique extension of $\mu$ to $\JJ_\mu$.

\begin{lem} \label{full range}
Let $\mu$ be a non--atomic measure on $\alg_c$ and
$\wh{\mu}$ its (unique) extension to $\JJ_\mu$. For
every $x\in \JJ_\mu$ and every $0<\varepsilon <\mu(x)$ there is 
$j\in \JJ_\mu$ such that $j\leq x$ and $\wh{\mu}(j) = \varepsilon$.
\end{lem}

\begin{proof}
Assume without loss of generality that
$x=\textbf{1}$. Since $\mu$ is non-atomic on $\alg_c$, 
it is easy to see that there is a $\mu$--good
partition $(a_n)_n$ of $\textbf{1}$
such that
\[ \sum_{n \in N} \mu(a_n) = \varepsilon, \]
where $N$ is the set of odd numbers. Then
$j = \bigvee_{n\in N} a_n \in \JJ_\mu$ and, of
course, $\wh{\mu}(j)=\varepsilon$.
\end{proof}

We will use this fact to prove the following
theorem.

\begin{thm} \label{jordan}
For every non--atomic measure $\mu$ on the Cantor
algebra, the $\mu$-Jordan extension algebra $(\JJ_\mu, \mu)$ is (metrically) isomorphic to
$(\JJ,\lambda)$.
\end{thm}

\begin{proof}
Let $\mu$ be a measure on $\alg_c$, and let
$\wh{\mu}$
be its extension to $\JJ_\mu$.
We identify the $\alg_c$'s generating tree $\TTT$ with $\cant$ here.
We will find a
subalgebra $\alg'$ of $\JJ_\mu$ such that
\begin{itemize}
 	\item there is an isomorphism $\varphi\colon \alg_c
 		\to \alg'$;
 	\item $\alg'$ is dense in $\JJ_\mu$;
 	\item for every $s\in 2^n$ we have
 		$\wh{\mu}(\varphi(s)) = 1/2^n$.
\end{itemize}

Fix an enumeration $\alg_c\setminus\{\emptyset\}=\{d_n\colon n\in\om\}$. We will
inductively define an isomorphism $\varphi$
between $\TTT$ and a subset of $\JJ_\mu$. 
Let $\varphi(\emptyset)=\textbf{1}$ and let $m_0=0$.
Assume that we have defined $\varphi(s)$ for every
$s\in 2^i$, $i\leq m_n$ in such a way that
$\{\varphi(s)\colon s\in 2^{m_n}\}$ is a partition
of $\mathbf{1}$ which is dense under  
$\{d_1,\dots,d_n\}$ and $\wh{\mu}(\varphi(s))=1/2^i$ for
$s\in 2^i$.

For some  $s\in 2^{m_n}$ we have $\varphi(s) \land d_{n+1} \ne
\textbf{0}$. Using Lemma \ref{full range} we can find an
element $y$ of $\JJ_\mu$ such that
\begin{itemize}
 	\item $\wh{\mu}(y) = 1/2^l$ for some $l>m_n$;
 	\item $y \leq \varphi(s) \land d_{n+1}$.
\end{itemize}
Now define $\varphi$ on $2^l$ in
such a way that 
\begin{itemize}
 	\item  $\varphi(t)=y$ for some $t\in 2^l$
 		extending $s$;
 	\item $\wh{\mu}(\varphi(s)) = 1/2^l$ 
        for $s\in 2^l$;
 	\item $\{\varphi(s)\colon s\in 2^l\}$ is a
 		partition of $\mathbf{1}$ refining
 		$\{\varphi(s)\colon s\in 2^{m_n}\}$. 
\end{itemize}
Put $m_{n+1}=l$. Define $\alg'$ to be the algebra generated by 
$\{\varphi(s)\colon s\in\cant\}$ and notice that
$\alg'$ is dense in $\alg$ and
$\wh{\mu}(\varphi(s))=1/2^{|s|}$ for $s\in \cant$.
Now, consider the Jordan extension
$\JJ'_{\wh{\mu}}$  of $\alg'$. 
\medskip

CLAIM: $\JJ_\mu = \JJ'_{\wh{\mu}}$. 
\medskip

It is enough to show that every $\mu$-good partition 
of $\textbf{1}$ into elements of the generating
tree $\TTT$ can be refined
to a $\mu'$--good partition of $\textbf{1}$ into elements of 
$\{\varphi(s)\colon s\in\cant\}$ and vice versa.

Let $\{a_n\colon n\in\om\}\sub \cant$ be a
$\mu$--good partition of $\textbf{1}$. Then, using
the fact that $\alg'$ is dense under $\alg$, for 
every $n$ we can find a $\mu'$--good partition $\{b^n_m\colon
m\in\om\} \sub \{\varphi(s)\colon s\in \cant\}$ of $a_n$. Thus, $\{b^n_m\colon
m,n\in\om\}\sub \{\varphi(s)\colon s\in \cant\}$ is a
$\mu'$--good partition of $\textbf{1}$. 

The reverse implication can proved in the same
way. $\Box_{\mbox{{\em Claim}}}$
\medskip

Now, we want to show that there is a metric
isomorphism $\psi\colon \JJ_\mu \to \JJ$.
For $s\in \TTT$ define
\[ \psi( \varphi(s) ) = s. \]
Of course $\psi = \varphi^{-1}$ so it is
an isomorphism. Additionally
$\wh{\mu}(s) =
1/2^{|s|} = \lambda(s)$, so $\psi$ is measure
preserving. We can extend $\psi$ to 
\[ \psi\colon \JJ'_{\wh{\mu}} \to \JJ. \] 
But $\JJ'_{\wh{\mu}} = \JJ_\mu$, so we are
done.
\end{proof}

The Maharam theorem implies that every
complete Boolean algebra supporting a non--atomic $\sigma$--additive
separable measure is metrically isomorphic to the Random
algebra with the Lebesgue measure. 
Theorem \ref{jordan} gives us an analogous
theorem for uniformly regular measures. There is
no \emph{complete} Boolean algebra supporting a uniformly
regular measure, but we can see a property
of being isomorphic to the Jordan algebra as a property of 
\emph{being as close to completeness as possible
without loosing uniform regularity}.

Theorem \ref{jordan} allows us to complete the proof of Theorem \ref{characterization}.

\begin{proof}(of Theorem \ref{characterization})
Assume $\alg$ supports a non--atomic uniformly regular measure $\mu$.
By Proposition \ref{Cohen} $(\alg,\mu)$ is metrically isomorphic to a 
subalgebra of $\JJ_\mu$ with $\wh{\mu}$, the unique extension of $\mu$ to $\sigma(\alg)$. 
By Theorem \ref{jordan} $\JJ_\mu$ is metrically isomorphic to $(\JJ,\lambda)$, so $(\alg, \mu)$ is metrically isomorphic
to a subalgebra of $\JJ$ with $\lambda$. 

Assume now that a Boolean algebra $\alg$ is a subalgebra of $\JJ$ and that it has a dense Cantor subalgebra $\algb$.
Then $\algb$ is uniformly $\lambda$--dense in $\alg$, so $\alg$ supports a uniformly regular measure.
\end{proof}

Because of the complexity of non $\sigma$--complete Boolean algebras, 
it seems that every characterization here inevitably has to 
involve subalgebras, like in Theorem \ref{characterization}.
So, we can only hope to fully characterize Boolean algebras supporting measures which are
\emph{maximal} with respect to some property. 

The fact analogous to the first part of Theorem \ref{characterization} for separable measures is an
easy application of Maharam theorem.

\begin{prop} \label{Mahar}
Assume a Boolean algebra $\alg$ supports a non--atomic separable measure $\mu$. 
Then $(\alg, \mu)$ is metrically isomorphic to a subalgebra of the Random algebra with the Lebesgue measure.
\end{prop}
\begin{proof}
Consider the (unique) extension $\wh{\mu}$ of $\mu$ to $\algb = \Borel(\mathrm{Stone}(\alg))$. 
It is non--atomic and separable, so the measure $\wh{\mu}$ defined on $\algb/\{b\colon \wh{\mu}(b)=0\}$ is, by the Maharam theorem, 
metrically isomorphic to the Lebesgue measure on the Random algebra. Since $\mu$ is strictly positive, 
the function $\varphi\colon \alg \to \algb/\{b\colon \wh{\mu}(b)=0\}$ defined by
\[ \varphi(a) = [a]_{\wh{\mu}} \]
is a monomorphism. 
\end{proof}

However, we cannot conclude from the above that every subalgebra of the Random algebra
supports a non--atomic separable measure (${\rm Free(\om_1)}$ is one of the counterexamples).
It is also unclear what assumption (analogous to containing a dense Cantor subalgebra in Theorem \ref{characterization}) 
should be added to Proposition \ref{Mahar} to obtain a characterization of algebras supporting non--atomic separable measures.
Thus, uniform regularity seems to be a more convenient notion for our purposes. 

\section{Higher cardinal versions of uniform regularity}\label{higher}

The positive classification results that we obtained when replacing separability by uniform regularity motivate us to consider higher cardinal versions of uniform regularity. The results that we have on this subject are preliminary and the further research is planned in the future. Nevertheless, the partial results we do have seem worth mentioning in this brief section.

For a measure $\mu$ on
a Boolean algebra $\algb$ we can define the cardinal number
\[
\ur(\mu)=\min\{\kappa:\mbox{there is a uniformly }\mu-\mbox{dense family }\FF\subseteq\algb
\mbox{ with }|\FF|=\kappa\}.
\]
This is a well defined cardinal invariant since $\algb$ is uniformly $\mu$-dense in $\algb$. Moreover, if the algebra is atomless, this cardinal invariant is always an infinite cardinal. In relation with known cardinal invariants of Boolean algebras
and measures, notice that clearly ${\rm ur} (\mu)$ is $\ge$ than the Maharam type of $\mu$ and also than the pseudoweight $\pi(\algb)$, which is defined as the smallest cardinality of a set $A$ of positive elements in $\algb$ such that for all $b\in
\algb \setminus\{ {\bf 0}\}$ there is
$a\in A$ such $a\le b$.

\begin{lemma}\label{homog} Suppose that $\mu$ is a non-atomic strictly positive measure on a Boolean algebra $\algb$. Then there is a partition of unity $\{a_n\colon \,n<\omega\}$ such that for every
$n$ and every ${\bf 0}\neq b\le a_n$ we have that $\ur(\mu\, |\, b)=\ur(\mu\, | \,a_n)$.
\end{lemma}

\begin{proof} Let us call an element $a\in B$ {\em uniform} if for every
	${\bf 0}\neq b\le a$ we have that $\ur(\mu\, |\, b)=\ur(\mu\, | \,a)$. We claim that the set of uniform elements is dense in $\algb$. If not, we can find a sequence $\{b_n\colon \,n<\omega\}$
of elements of $\algb$ such that $\ur(\mu\, |\, b_n)>\ur(\mu\, | \,b_{n+1})$, which gives an infinite decreasing sequence of cardinals. Now it suffices to take a maximal antichain in the family of uniform elements. Such an antichain must be countable, by the ccc property of the algebra, and it must be a
partition of unity by the density of the set of uniform elements.
\end{proof}

The above Lemma parallels the reduction in Maharam's theorem to homogeneous measures.
We shall say that a measure $\mu$ on an algebra $\algb$ is {\em uniformly $\kappa$--regular} if
${\bf 1}$ is a uniform element and $\ur(\mu)=\kappa$. We would like to classify pairs $(\algb,\mu)$ where $\mu$ is a uniformly $\kappa$-regular
measure on $\algb$, for various $\kappa$. For $\kappa=\aleph_0$ we have already determined that
these pairs are exactly subalgebras of the Jordan algebra with the Lebesgue measure. 

We need to introduce higher analogues of the Jordan algebra.


Let us denote by $\lambda_\kappa$ the usual product measure on $2^\kappa$ and by
$\alg_\kappa$ the clopen algebra of  $2^\kappa$. We shall also denote by $\TT_\kappa$
the family of basic clopen sets in $2^\kappa$, which can be identified with the free algebra on
$\kappa$ generators.
Then Kakutani's theorem says that $\lambda_\kappa$ is obtained as an extension of the measure
on $\TT_\kappa$ which to each basic clopen set of the form $[s]$ assigns $1/2^{|{\rm dom}(s)|}$.
The point is that all Borel sets are indeed measurable when we extend the measure, even though the
$\sigma$-completion of $\alg_\kappa$ only gives us the Baire algebra $\algb_\kappa$, which in general does not contain all Borel subsets of $2^\kappa$. 
We obtain the same algebra if we work with $[0,1]^\kappa$ in place of $2^\kappa$, and we denote both of these measures as $\lambda_\kappa$.

Now we shall define analogues of the Jordan algebra. For any $\kappa$ let 
$\JJ^\kappa$  be the algebra of
those open sets in $[0,1]^\kappa$ whose boundaries have the $\lambda_\kappa$ measure $0$, so
$\JJ^\omega=\JJ$. In other words, $\JJ^\kappa = \JJ_{\lambda_\kappa}({\rm Free}(\kappa))$. We might hope to prove that a  Boolean algebra supports a uniformly $\kappa$--regular measure if and only if it is isomorphic to a subalgebra of the algebra $\JJ^\kappa$ containing a dense copy of some fixed nice algebra $\alg_\kappa^\ast$. Note however that certainly this hypothetic
$\alg_\kappa^\ast$ cannot be $\alg_\kappa$, since as soon as a Boolean algebra contains a copy of $\TT_\kappa$,
it has an independent sequence of length $\kappa$ and hence it supports a measure of type $\kappa$.
However, for $\kappa>\omega$ there may be Boolean algebras $\algb$ that do not support a measure of type $\kappa$ but all measures $\mu$ on $\algb$ satisfy ${\rm ur}(\mu)\ge\kappa$, see Section
\ref{separable} for examples. It is in fact more reasonable to widen our acceptance criterion from a fixed 
$\alg_\kappa^\ast$ to a quotient of $\alg_\kappa$, as suggested by the following observation which generalizes Proposition \ref{Cohen}. 

\begin{prop}\label{Cohenkappa} If a Boolean algebra supports a uniformly $\kappa$-regular (non-atomic) measure, then it is isomorphic to a subalgebra of the $\sigma$-completion of a quotient of ${\rm Free}(\kappa)$.
\end{prop}
\begin{proof} Let $\alg$ be a Boolean algebra that supports a uniformly $\kappa$-regular measure and let $\algb$ be a uniformly $\mu$-dense subset of $\alg$ size $\kappa$. We may assume that $\algb$ is a Boolean algebra. Therefore $\algb$ is
	isomorphic to a quotient $\algb'$ of ${\rm Free}(\kappa)$. Let $\algc'$
be the $\sigma$-completion of $\algb'$, we shall show that $\alg$ is isomorphic to a subalgebra of 
$\algc'$. To do this, we use exactly the same definition of $f$ as in the proof of Proposition
\ref{Cohen} and  the proof that $f$ is a homomorphism remains the same.
\end{proof}

If we wish to generalise further, we need to develop the analogues of the uniqueness of the Jordan
algebra. The proof we had in the separable case rested upon the uniqueness of the Cantor algebra. 
In the higher-dimensional case we cannot hope for that, but perhaps we can obtain uniqueness restricted to algebras that have the same dense set. 
This research brings us out of the scope of the present article
and we plan it for the future work.

\section{Separability versus uniform regularity} \label{separable} 

In Section \ref{uniform regularity} we presented a
characterization of Kelley algebras supporting
uniformly regular measures. Finding a characterization of Kelley algebras
supporting separable measures seems to be a more
difficult task (see e.g. \cite{DzP}) as well as finding a characterization of Kelley algebras
carrying only separable measures. The latter is at least possible under $\mathsf{MA}\mathrm{(\om_1)}$:

\begin{thm}{\em (\cite[Theorem 9]{Fremlin})} \label{Fremlin-char}
If a Boolean algebra carries only separable measures, then it is small.
Under $\mathsf{MA}\mathrm{(\om_1)}$ the converse implication holds.
\end{thm}

However, consistently there are small Boolean algebras carrying a non--separable measure
(see e.g. the literature listed in \cite{Todorcevic} after Theorem 6.4). 

The natural question is if one can
use the result from the previous section to get
some information about properties of Kelley
algebras supporting only separable measures.
Some connections of uniform
regularity and separability are obvious: e.g.
uniformly regular measures are separable. 
The following fact indicates that 
there are some more subtle relationships at work.

\begin{thm}{\em (\cite[Theorem 4.6]{Pbn})}\label{alternative}
Every Boolean algebra carries either a
non--separable measure or a measure which is
uniformly regular. \end{thm}

Recently, Miko\l aj Krupski proved the above
theorem in a more general setting, see
\cite{Krupski}.

One should point out here that (consistently) there are
small Boolean algebras without uniformly regular
measures. Indeed, Talagrand (\cite{Talagrand})
used $\mathsf{CH}$ to construct a small Gronthendieck space $K$, i.e.
a space such that there are no non--trivial
(i.e. not weakly convergent) weakly$^*$ convergent
sequences of measures on $K$. Grothendieck property 
implies that $P(K)$ does not have
$G_\delta$ points, and thus (by \cite[Proposition 2]{Pol}) $K$
does not carry a uniformly regular measure.   
Talagrand's example is zero--dimensional
and cannot be continuously mapped onto
$[0,1]^{\om_1}$. Thus, it is the Stone space of a
small Boolean algebra without a uniformly regular
measure. Such an example cannot be, however,
constructed without additional axioms, because of 
Theorem \ref{Fremlin-char} and Theorem \ref{alternative}.

Notice also, that the alternative in Theorem
\ref{alternative} is, by no means, exclusive.
There are many Boolean algebras with both
non--separable and uniformly regular measures
(the Jordan algebra can serve as an example here).

For our purposes a \emph{strictly positive}
version of Theorem \ref{alternative} would be most
desirable. However, it turned out that we cannot
hope for that:

\begin{thm} \label{bellike}
There is a Kelley algebra supporting only
separable measures but no uniformly regular one.
\end{thm}

We will prove this theorem building on ideas contained
in \cite{Bell}. Bell constructed in this paper 
a Boolean space which is separable, which does not 
have a countable $\pi$--base and whose algebra of
clopen subsets is small. 

The space presented below is
similar to the space constructed by Bell. 
However, our approach is different and, at least for our purposes,
simpler than that of Bell. In particular,
it allows us to prove that each measure supported by this
space is separable. 

First, we introduce some
notation. For $A \sub \om$ let $A^0 \sub 2^\om$ be
the set of the form
\[ A^0 = \{x\in 2^\om\colon \forall n\in A \ \
x(n) = 0\}. \]
For a family $\AAA \sub P(\om)$ let 
\[ \AAA^0 = \{A^0\colon A\in \AAA\}. \]
If $\AAA \sub P(\om)$, then let
$\alg(\AAA)\sub P(2^\om)$ be the Boolean algebra generated by
$\AAA^0$ and let $K(\AAA)$ be the Stone space of
this algebra. For $A \sub \om$ let 
\[ A^1 = \{x\in 2^\om\colon \exists n\in A \ \
x(n) = 1\}. \]
Of course $A^1 = (A^0)^c$ for every $A\sub \om$.

Let us collect some immediate observations:
\begin{prop} \label{basic facts}
\begin{itemize}
	\item[1)] If $\mathrm{Fin}$ is the set of finite subsets of
		$\om$, then $\mathrm{Fin}^0$ generates $\mathrm{Clop}(2^\om)$;
	\item[2)] If $\mathrm{Fin} \sub \AAA$, then $\alg(\AAA)$ is
		an extension of the Cantor algebra (and there is a continuous function
		from $K(\AAA)$ onto $2^\om$). Every $x\in
		2^\om$ can be interpreted as a closed subset of
		$K(\AAA)$. Namely, for $x\in 2^\om$ let $F_x$ be the set of all
		ultrafilters on $\alg(\AAA)$ extending the filter
		generated by 
		\[ \{\{n\}^0\colon x(n)=0 \} \cup \{\{n\}^1
		\colon x(n)=1\}; \]
\end{itemize}
\end{prop}

For $\AAA\sub P(\om)$ the family of elements of
the form
\[ A^0_0\cap A^0_1 \cap \dots \cap A^0_k \cap
A_{k+1}^1 \cap \dots \cap A_n^1 \]
is a base of $K(\AAA)$. Every element of
$\alg(\AAA)$ is a finite union of sets of this
form. Since $A^0\cap B^0 = (A\cup B)^0$, if $\AAA$
is closed under taking finite unions, then elements
of the above base can be written in a simpler
form:
\[ A^0 \cap A^1_0 \cap \ldots \cap A^1_n \]
for $A, A_0, \ldots, A_n \in \AAA$.

Before pointing out which particular family $\AAA$
we will consider, we prove two general theorems
concerning spaces $K(\AAA)$.

\begin{prop} \label{separability}
	Let $\mathrm{Fin} \sub \AAA\sub P(\om)$. Then there is a
countable family of closed subsets $\FF$ of $K(\AAA)$ 
such that for every nonempty open set $U$ in
$K(\AAA)$ there is $F\in \FF$ such that $F\sub U$.
Consequently, $K(\AAA)$ is separable.
\end{prop}
\begin{proof}
Let
\[ \FF = \{F_x\colon x\in 2^\om, \ x(n)=1\mbox{ for
finitely many }n\mbox{'s}\}, \] 
where $F_x$ is as in Proposition \ref{basic facts} (2).
Let $U \sub K(\AAA)$ be an open subset. 
Without loss of generality we can assume that it
is of the form 
\[ U  =  A^0_0\cap A^0_1 \cap \ldots \cap A^0_k \cap
A_{k+1}^1 \cap \ldots \cap A_n^1, \]
for $A_1, \dots, A_n \in \AAA$. If $U$ is nonempty, then for
every $i>k$ we have 
\[ B_i = A_i \setminus (A_0\cup \dots \cup A_k) \ne
\emptyset. \]
Pick $n_i \in B_i$ for every $k<i\leq n$. Let $x\in 2^\om$
be such that $x(n_i)=1$ for every $i$ and $x(n)=0$
if there is no $i$ such that $n=n_i$. Then $F_x\in
\FF$ and $F_x\sub U$. 
\end{proof}

\begin{prop} \label{pi-base} 
	Let $\mathrm{Fin} \sub \AAA \sub P(\om)$ and assume that
$\AAA$ is closed under finite unions.
Suppose that $\AAA$ does not have a cofinite family of
cardinality $\lambda$, i.e. for every $\AAA_0\sub
\AAA$ of size $\lambda$ there is $B\in
\AAA$ such that $B\setminus A \ne \emptyset$
for every $A\in \AAA_0$. Then $K(\AAA)$ does not have
a $\pi$--base of size $\lambda$.
\end{prop}
\begin{proof}
Suppose $\VV$ is a $\pi$--base of $K(\AAA)$. 
We can assume that it consists of sets of the form 
\[ V = A_V^0 \cap A^1_0 \cap \ldots
\cap A^1_n \]
for $A_V, A_0, \dots, A_n \in \AAA$. 
Assume that $B\in \AAA$ is infinite and $V\in \VV$ is
such that $V \sub B^0$. Clearly, $B\sub A_V$.
So, if we let $\AAA_0 = \{A_V\colon V\in \VV\}$,
then
\begin{itemize}
	\item $|\AAA_0| \leq |\VV|$,
	\item for every $B\in \AAA$ there is
		$A\in \AAA_0$ such that $B \sub A$, so $\lambda < |\AAA_0|$.
\end{itemize}
Therefore, $\lambda<|\VV|$.
\end{proof}
\bigskip

Notice that if a family $\AAA$ contains an
uncountable pairwise almost disjoint family
$(A_\alpha)_{\alpha<\om_1}$, then
$(A^0_\alpha)_{\alpha<\om_1}$ forms an uncountable
independent sequence in $\alg(\AAA)$ and
consequently $\alg(\AAA)$ is big. So, if we want to construct a
Kelley algebra which is small (to omit supporting a non--separable measure, cf. Theorem \ref{Fremlin-char}), 
we have to use a family
which does not contain many pairwise almost disjoint
sets. 

The natural example of such a family satisfying
also the conditions of Theorem \ref{pi-base} is
the following.
Let $\{T_\alpha\colon \alpha<\om_1\} \sub P(\om)$
be such that $T_0 =
\emptyset$ and for every $\alpha < \beta <\om_1$
the set $T_\alpha \setminus T_\beta$ is
finite and $T_\beta \setminus T_\alpha$ is
infinite. Shortly speaking, $(T_\alpha)_\alpha$ is a
strictly $\sub^*$--increasing tower. Let $\TTT$ consists of those sets $T\sub \om$ such
that $T= T_\alpha \cup F$ for some $\alpha<\om_1$
and some finite $F\sub \om$. 

Notice that $\TTT$ satisfies the assumptions of
Proposition \ref{pi-base}, so $K(\TTT)$ is a
separable space without a countable $\pi$--base.
Now we will prove that all measures on $K(\TTT)$
are separable.

Recall that if $\mu$ is a non--separable measure
on a Boolean algebra $\alg$, then we can find an
uncountable family $\AAA$ of generators of $\alg$
and $\varepsilon > 0$ such that 
\[ \mu(A\bigtriangleup B) > \varepsilon \]
for every distinct $A$, $B\in \AAA$. Otherwise, we
could find a countable family $\BB$ which is
$\mu$--dense in the set of generators of $\alg$. 
But then the (countable) Boolean algebra generated by
$\BB$ would be $\mu$--dense in  $\alg$, 
and so $\mu$ would be separable. 

\begin{thm} \label{bell}
Every measure on $\alg(\TTT)$ is separable.
\end{thm}
\begin{proof}
Suppose toward a contradiction that there is a
non--separable measure $\mu$ on $\alg(\TTT)$.
Then using the above remark assume that
\[ \mu(T_\alpha^1 \bigtriangleup T_\beta^1) >
\varepsilon \]
for every $\alpha < \beta <\om_1$. We can do it since
$\{T^1\colon T\in \TTT\}$ generates $\alg(\TTT)$
and we can consider a subalgebra of $\alg(\TTT)$
if necessary.

For $\alpha<\om_1$ denote \[ \rho(T_\alpha) =
\sup\{\mu(F^1)\colon
F\in \mathrm{Fin}, F\sub T_\alpha\}.\] Notice that
$\rho(T_\alpha) = \nu_*(T^1_\alpha)$, where $\nu =
\mu|\alg(\mathrm{Fin}^1)$, and that $\rho(T_\alpha)\leq
\mu(T^1_\alpha)$ for every $\alpha$.

Considering a sub--tower of
$(T_\alpha)_{\alpha<\om_1}$ of height $\om_1$ 
if necessary, we can assume that $|\mu(T_\alpha^1) -
\mu(T_\beta^1)| < \varepsilon/3$ and $\rho(T_\alpha) >
\sup_\alpha \rho(T_\alpha) - \varepsilon/6$
for every $\alpha$, $\beta<\om_1$.
\medskip

CLAIM 1. We can assume that the tower
$(T_\alpha)_{\alpha<\om_1}$ has the following
property (*): for every finite $F\sub \om$ either
there is no $\alpha<\om_1$ such that $F\sub
T_\alpha$ or $F\sub T_\alpha$ for uncountably many
$\alpha$'s. 
\medskip

Enumerate $\{F_n\colon n\in \om\}$ the set of
those finite sets which are included in $T_\alpha$
for at most countably many $\alpha$'s. Let
$\alpha_n = \sup\{\alpha\colon F_n\sub T_\alpha\}$
and let $\gamma = \sup_n \alpha_n$. It is clear
that $(T_\alpha)_{\alpha>\gamma}$ is a tower of
height $\om_1$ with the property (*). So we can
assume without loss of generality that $\gamma=0$
and $(T_\alpha)_{\alpha<\om_1}$ has property (*).
\medskip

CLAIM 2. For every $\alpha<\beta$ we have
\[ \mu( (T_\alpha \cup T_\beta)^1 )> \mu(T^1_\beta)+
\varepsilon/3. \]
Indeed, 
\[ \varepsilon < \mu(T_\alpha^0 \bigtriangleup T_\beta^0) =
\mu(T_\alpha^0)+\mu(T_\beta^0) -
2\mu( (T_\alpha\cup T_\beta)^0 ).  \]
Hence
\[ 2\mu( (T_\alpha \cup T_\beta)^0 ) <
\mu(T_\alpha^0) + \mu(T_\beta^0) -\varepsilon \leq
2\mu(T_\beta^0) + \varepsilon/3 - \varepsilon \]
and
\[ \mu( (T_\alpha \cup T_\beta)^0 ) <
\mu(T_\beta^0) - \varepsilon/3. \]
Since $(T_\alpha \cup T_\beta)^1 = ((T_\alpha \cup
T_\beta)^0)^c$ we have
\[ \mu( (T_\alpha \cup T_\beta)^1 ) >
\mu(T^1_\beta) + \varepsilon/3. \]

CLAIM 3. If $\alpha<\beta$ and $F=T_\alpha\setminus
T_\beta$, then for every nonempty finite $G\sub
T_\beta$ we have $\mu( (F\cup G)^1
)>\mu(G^1)+\varepsilon/3$.
\medskip

Using Claim 2 we have
\[ \mu(T^1_\beta)+\mu( F^1
) - \mu(T_\beta^1 \cap F^1) = \mu(T_\alpha^1 \cup T_\beta^1) = \mu( (T_\alpha \cup T_\beta)^1 )
> \mu(T^1_\beta) + \varepsilon/3. \]
Thus
\[ \mu(F^1) - \mu(T^1_\beta \cap F^1) >
\varepsilon/3\]
and for every finite $G\sub T_\beta$
\[ \mu(F^1) > \mu(G^1\cap F^1) + \varepsilon/3,
\ \ \ \mbox{ so } \ \ \
 \mu(F^1 \setminus G^1) > \varepsilon/3. \]
Finally 
\[ \mu( (F\cup G)^1 ) = \mu(F^1 \cup G^1) >
\mu(G^1)+\varepsilon/3, \]
and Claim 3 is proved.
\medskip

Now, let $G\sub T_0$ be a finite set such that
$\mu(G^1)>\rho(T_0)-\varepsilon/6$. Because of
property (*) there is $\beta<\om_1$ such that
$G\sub T_\beta$. Let $F=T_0\setminus T_\beta$.
Then, using Claim 3 we have
\[ \mu( (F\cup G)^1 ) > \mu(G^1) + \varepsilon/3 >
\rho(T_0) - \varepsilon/6 + \varepsilon >
\sup_\alpha \rho(T_\alpha) -2\varepsilon/6 +
\varepsilon/3 = \sup_\alpha \rho(T_\alpha). \]
But $F\cup G \sub T_0$, so $\mu( (F\cup G)^1 )\leq
\rho(T_0)$, a contradiction.
\end{proof}

Now we are finally ready to present a proof of
Theorem \ref{bellike}:

\begin{proof}
Since $K(\TTT)$ is separable, $\alg(\TTT)$ 
supports a strictly positive measure. Since $K(\TTT)$ does
not have a countable $\pi$--base, it cannot
support a uniformly regular measure. Finally, Theorem
\ref{bell} implies that all measures on
$\alg(\TTT)$ are separable.
\end{proof}

We finish this section with two remarks, not
connected directly to measures. First, we show that the 
fact that $\alg(\TTT)$ does not have an uncountable 
independent sequence can be deduced in a slightly simpler way.

\begin{thm}
The Boolean algebra $\alg(\TTT)$ is small.
\end{thm}
\begin{proof}
We will
prove that every filter on $\alg(\TTT)$ can be extended to an
ultrafilter by countably many sets. 
This would imply that every closed subset of
$K(\TTT)$ contains a point with a
relatively countable $\pi$--character and, 
because of Shapirovsky's theorem (see 
\cite[Theorem 6.1]{Todorcevic}) that $\alg(\TTT)$ is small. 

First, observe that if 
$\AAA \sub P(\om)$,
$\FF$ is a filter on $\alg(\AAA)$ and there is
no $A\in \AAA$ such that $\FF$ can be extended by
$A^0$, then $\FF$ is an ultrafilter. Of course, the same holds true for the sets of the form $A^1$.

Let $\FF$ be a filter on $\alg(\TTT)$.
Notice that without loss of generality we
can assume that $\{n\}^0\in \FF$ or $\{n\}^1\in
\FF$ for every $n\in \om$, extending $\FF$ by
at most countably many sets, if necessary. 

For $T\in \TTT$ say that the level of $T$ is
$\alpha$  if $T=^*T_\alpha$ (denote it by $\mathrm{lv}(T)=\alpha$). 
Assume $\FF$ is not an ultrafilter
and let $\gamma$ be the minimal number
such that there is $S$ of level $\gamma$ such that
$\FF$ can be extended by $S^1$. Notice that since $S^1\notin
\FF$ the set $\{n\}^0\in \FF$ for every $n\in S$. 

Extend $\FF$ to $\FF'$ by $S^1$. Then extend
$\FF'$ to $\FF''$ by countably many sets in such a
way that $\FF''$ cannot be extended by any element
of the set 
\[ \{T^0\colon \mathrm{lv}(T)\leq \gamma\}. \]
It can be done since the above set is countable.

We will show that $\FF''$ is an ultrafilter by
showing that it cannot be extended by a set $T^0$
for $T\in \TTT$. Indeed, let $T\in \TTT$. If
$\mathrm{lv}(T)\leq \gamma$, then either $T\in \FF''$ or
$\FF''$ cannot be extended  by $T$. If
$\mathrm{lv}(T)>\gamma$, then the set $S\setminus T$ is
finite. Moreover $(S\setminus T)^0 \in \FF$. 
So, the set $T^0\cap (S\setminus T)^0 \supseteq S^0$
cannot be added to $\FF''$, and thus $\FF''$ cannot
be extended by $T^0$. 
\end{proof}

The following theorem is also worth
mentioning in this context. A Boolean algebra is
minimally generated if there is a maximal chain in
the lattice of its subalgebras and this chain is
well--ordered (for a more intuitive but longer
definition, see e.g. \cite[Section 0]{Koppelberg}). The class of
minimally generated Boolean algebras is quite a big
sub--class of small Boolean algebras. In
\cite{Koppelberg} Koppelberg posed even the
question if those classes are equal (she answered
it negatively in \cite[Example 1]{Koppelberg2}). Minimally generated Boolean algebras are
interesting from the point of view of measure theory because they carry only separable measures.
Moreover, the following theorem holds.

\begin{thm} {\em (\cite[Theorem 4.15]{Pbn})}\label{second}
Every minimally generated Kelley algebra
supports a uniformly regular measure.
\end{thm}

It follows that the Boolean algebra constructed
above is an another example of a small Boolean
algebra which is not minimally generated.

\bibliographystyle{amsplain}
\bibliography{class}

\providecommand{\bysame}{\leavevmode\hbox to3em{\hrulefill}\thinspace}
\providecommand{\MR}{\relax\ifhmode\unskip\space\fi MR }
\providecommand{\MRhref}[2]{%
  \href{http://www.ams.org/mathscinet-getitem?mr=#1}{#2}
}
\providecommand{\href}[2]{#2}
\begin{thebibliography}{10}

\bibitem{Bell}
M.~G. Bell, \emph{{$G_\kappa$} subspaces of hyadic spaces}, Proc. Amer. Math.
  Soc. \textbf{104} (1988), no.~2, 635--640. \MR{962841 (90a:54062)}

\bibitem{Pbn}
P.~Borodulin-Nadzieja, \emph{Measures on minimally generated {B}oolean
  algebras}, Topology Appl. \textbf{154} (2007), no.~18, 3107--3124.
  \MR{2364639 (2009c:28008)}

\bibitem{Camerlo}
R.~Camerlo and S.~Gao, \emph{The completeness of the isomorphism relation for
  countable {B}oolean algebras}, Trans. Amer. Math. Soc. \textbf{353} (2001),
  no.~2, 491--518. \MR{1804507 (2001k:03097)}

\bibitem{Day}
G.~W. Day, \emph{Superatomic {B}oolean algebras}, Pacific J. Math. \textbf{23}
  (1967), 479--489. \MR{0221993 (36 \#5045)}

\bibitem{DzP}
M.~D{\v{z}}amonja and G.~Plebanek, \emph{Strictly positive measures on
  {B}oolean algebras}, J. Symbolic Logic \textbf{73} (2008), no.~4, 1416--1432.
  \MR{2467227 (2010b:03077)}

\bibitem{Effros}
E.~G. Effros, \emph{Transformation groups and {$C^{\ast} $}-algebras}, Ann. of
  Math. (2) \textbf{81} (1965), 38--55. \MR{0174987 (30 \#5175)}

\bibitem{Fremlin}
D.~H. Fremlin, \emph{On compact spaces carrying {R}adon measures of uncountable
  {M}aharam type}, Fund. Math. \textbf{154} (1997), no.~3, 295--304.
  \MR{1475869 (99d:28019)}

\bibitem{Harvey}
H.~Friedman and L.~Stanley, \emph{A {B}orel reducibility theory for classes of
  countable structures}, J. Symbolic Logic \textbf{54} (1989), no.~3, 894--914.
  \MR{1011177 (91f:03062)}

\bibitem{Gaifman}
H.~Gaifman, \emph{Concerning measures on {B}oolean algebras}, Pacific J. Math.
  \textbf{14} (1964), 61--73. \MR{0161952 (28 \#5156)}

\bibitem{Gao-invariant}
S.~Gao, \emph{Invariant descriptive set theory}, Pure and Applied Mathematics
  (Boca Raton), vol. 293, CRC Press, Boca Raton, FL, 2009. \MR{2455198
  (2011b:03001)}

\bibitem{Goncharov}
S.~S. Goncharov, \emph{Countable {B}oolean algebras and decidability}, Siberian
  School of Algebra and Logic, Consultants Bureau, New York, 1997. \MR{1444819
  (98h:03044b)}

\bibitem{HaKeLo}
L.~A. Harrington, A.~S. Kechris, and A.~Louveau, \emph{A {G}limm-{E}ffros
  dichotomy for {B}orel equivalence relations}, J. Amer. Math. Soc. \textbf{3}
  (1990), no.~4, 903--928. \MR{1057041 (91h:28023)}

\bibitem{HornTarski}
A.~Horn and A.~Tarski, \emph{Measures in {B}oolean algebras}, Trans. Amer.
  Math. Soc. \textbf{64} (1948), 467--497. \MR{0028922 (10,518h)}

\bibitem{Jordan}
C.~Jordan, \emph{Remarques sur les integrales definies}, J. Math. Pures Appl.
  \textbf{8} (1892), 69--99.

\bibitem{Kanovei}
V.~Kanovei, \emph{Borel equivalence relations}, University Lecture Series,
  vol.~44, American Mathematical Society, Providence, RI, 2008, Structure and
  classification. \MR{2441635 (2009f:03060)}

\bibitem{Kelley}
J.~L. Kelley, \emph{Measures on {B}oolean algebras}, Pacific J. Math.
  \textbf{9} (1959), 1165--1177. \MR{0108570 (21 \#7286)}

\bibitem{Koppelberg2}
S.~Koppelberg, \emph{Counterexamples in minimally generated {B}oolean
  algebras}, Acta Univ. Carolin. Math. Phys. \textbf{29} (1988), no.~2, 27--36.
  \MR{983448 (90a:06014)}

\bibitem{Koppelberg}
\bysame, \emph{Minimally generated {B}oolean algebras}, Order \textbf{5}
  (1989), no.~4, 393--406. \MR{1010388 (90g:06022)}

\bibitem{Krupski}
M.~Krupski, \emph{Regularity properties of measures on compact spaces},
  preprint.

\bibitem{Maharam}
D.~Maharam, \emph{On homogeneous measure algebras}, Proc. Nat. Acad. Sci. U. S.
  A. \textbf{28} (1942), 108--111. \MR{0006595 (4,12a)}

\bibitem{Scottish}
R.~D. Mauldin (ed.), \emph{The {Scottish Book}: {Mathematics} from the
  {Scottish Cafe}}, Birkhauser, 1982. \MR{666400}

\bibitem{Mercourakis}
S.~Mercourakis, \emph{Some remarks on countably determined measures and uniform
  distribution of sequences}, Monatsh. Math. \textbf{121} (1996), no.~1-2,
  79--111. \MR{1375642 (97j:28029)}

\bibitem{vonNeumann}
F.~J. Murray and J.~Von~Neumann, \emph{On rings of operators}, Ann. of Math.
  (2) \textbf{37} (1936), no.~1, 116--229. \MR{1503275}

\bibitem{Peano}
G.~Peano, \emph{Applicazioni geometriche del calcolo infinitesimale}, Fratelli
  Bocca, 1887.

\bibitem{Pol}
R.~Pol, \emph{Note on the spaces {$P(S)$} of regular probability measures whose
  topology is determined by countable subsets}, Pacific J. Math. \textbf{100}
  (1982), no.~1, 185--201. \MR{661448 (83g:54024)}

\bibitem{Talagrand}
M.~Talagrand, \emph{Un nouveau {C(K)} qui poss\`ede la propri\'et\'e de
  {G}rothendieck}, Israel J. Math. \textbf{37} (1980), no.~1-2, 181--191.
  \MR{599313 (82g:46029)}

\bibitem{Tarski}
A.~Tarski, \emph{Ideale in vollst\"andigen {M}engenk\"orpern. {II}}, Fund.
  Math. \textbf{33} (1945), 51--65. \MR{0017737 (8,193b)}

\bibitem{Todorcevic}
S.~Todorcevic, \emph{Chain-condition methods in topology}, Topology Appl.
  \textbf{101} (2000), no.~1, 45--82. \MR{1730899 (2001a:54055)}

\end{thebibliography}

\end{document}